\documentclass[12pt]{amsart}

\usepackage[margin=1in,letterpaper,portrait]{geometry}
\usepackage{amsmath}
\usepackage{amssymb}
\usepackage{amsthm}
\usepackage{verbatim}
\usepackage{url}
\usepackage{enumerate}
\usepackage[pdftex,bookmarksnumbered=true,linkbordercolor={1 0 0}]{hyperref}

\newcommand{\mb}{\mathbb}

\newcommand{\gauss}[2]{\genfrac{[}{]}{0pt}{}{#1}{#2}_q}

\begin{document}
\theoremstyle{plain} 
\newtheorem{Thm}{Theorem}[section] 
\newtheorem{Cor}[Thm]{Corollary}
\newtheorem{Main}{Main Theorem}
\renewcommand{\theMain}{} 
\newcommand{\Sol}{\text{Sol}}
\newtheorem{Lem}[Thm]{Lemma}
\newtheorem{Prop}[Thm]{Proposition} 

\theoremstyle{definition} 
\newtheorem*{Def}{Definition}
\newtheorem{Exer}{Exercise} 
\newtheorem{Prob}{Problem}
\newtheorem*{Rem}{Remark}
\newtheorem{Conj}{Conjecture} 

\theoremstyle{remark}      

\title{The Splitting Subspace Conjecture}
\author{Eric Chen \and Dennis Tseng}
\address{Eric Chen, Princeton University, Princeton, NJ 08544}
\email{ecchen@princeton.edu}
\address{Dennis Tseng, Massachusetts Institute of Technology, Cambridge, MA 02139}
\email{DennisCTseng@gmail.com}
\date{August 3, 2012}

\allowdisplaybreaks

\begin{abstract}
We answer a question by Niederreiter concerning the enumeration of a class of subspaces of finite dimensional vector spaces over finite fields by proving a conjecture by Ghorpade and Ram.

\end{abstract}

\maketitle

\section{Introduction}
We positively resolve the Splitting Subspace Conjecture, stemming from a question posed by Niederreiter (1995) \cite[p.~11]{NDR} and stated by Ghorpade and Ram \cite{ESS}. The conjecture was inspired by earlier work from Zeng, Han and He \cite{ZHH} and Ghorpade, Hasan and Kumari \cite{GHK} . We first define the notion of a $\sigma$-splitting subspace.

\begin{Def}
In the vector space $\mb{F}_{q^{mn}}$ over the finite field $\mb{F}_{q}$, given a $\sigma\in\mb{F}_{q^{mn}}$ such that $\mb{F}_{q^{mn}}=\mb{F}_{q}(\sigma)$, a ($m$-dimensional) subspace $W$ of $\mb{F}_{q^{mn}}$ is a $\sigma$-\emph{splitting subspace} if
\begin{align*}
W\oplus\sigma W\oplus \cdots\oplus\sigma^{n-1}W=\mb{F}_{q^{mn}}.
\end{align*}
\end{Def}

For example, $\left\{1,\sigma^{m},\sigma^{2m},\ldots,\sigma^{(n-1)m}\right\}$ spans a $\sigma$-splitting subspace. If $n=1$, then ${F}_{q^{mn}}$ is the only $\sigma$-splitting subspace; if $m=1$, then each 
$1$-dimensional subspace of ${F}_{q^{mn}}$ is $\sigma$-splitting.

\begin{Conj}[Ghorpade-Ram]
The number of $\sigma$-splitting subspaces is
\begin{align*}
\frac{q^{mn}-1}{q^{m}-1}q^{m(m-1)(n-1)}.
\end{align*}
\end{Conj}
\noindent
This follows as Corollary \ref{grssc} from our main result, Theorem \ref{alpha}. The next two sections are devoted to proving this theorem. 
We first construct a recursion that gives the cardinality of more general classes of subspaces, including the $\sigma$-splitting subspaces, and then solve this recurrence to obtain the result. Finally, we discuss some special cases of our more general result. 

\section{Recursion}
For the remainder of this article, unless otherwise noted, consider more generally the vector space $\mb{F}_{q^{N}}(=\mb{F}_{q}^{N})$ over the finite field $\mb{F}_{q}$, given a $\sigma\in\mb{F}_{q^{N}}$ such that $\mb{F}_{q^{N}}=\mb{F}_{q}(\sigma)$.

We begin by isolating the key property of the linear transformation $v\mapsto\sigma v$.

\begin{Prop}
\label{preserve}
The linear endomorphisms of $\mb{F}_{{q}^{N}}$ that preserve no subspaces other than $\{0\}$ and all of $\mb{F}_{{q}^{N}}$ are exactly those which act as multiplication by a primitive element $\sigma$ that generates the extension $\mb{F}_{q}(\sigma)=\mb{F}_{q^{N}}$.
\end{Prop}

\begin{proof}
Operators defined as multiplication by a primitive element $\sigma$ generating the extension $\mb{F}_{q}(\sigma)=\mb{F}_{q^{N}}$ cannot preserve any subspaces except $\{0\}$ and $\mb{F}_{{q}^{N}}$, for if $W$ is such a subspace with nonzero $w\in W$, then $w\sum_{i=0}^{N-1}{a_{i}\sigma^{i}}\in W,\ a_{i}\in \mb{F}_{q}$, so $W=\mb{F}_{q^{N}}$.
Conversely, note that any linear operator $T$ together with
the vector space $\mb{F}_{q}^{N}$ can be viewed as a finitely
generated $\mb{F}_{q}[x]$ module $M$, where $x$ acts as $T$. Since $\mb{F}_{q}[x]$ is a principal ideal domain, we can use the
primary decomposition of $M$ to find
$M\cong\bigoplus_{i=1}^{k}{\mb{F}_{q}[x]/(p_{i}(x)^{r_{i}})}$, where
$p_{i}$ is a polynomial for each $i$ and $r_{i}$ is a positive
integer. 
\begin{comment}This is equivalent to writing the generators for the relations of $M$ in a matrix ($M$ is finitely
presented since $\mb{F}_{q}[x]$ is Noetherian), reducing the matrix to
Smith Normal Form, and then applying the Chinese Remainder Theorem.
\end{comment}

If $T$ preserves no proper subspaces of $\mb{F}_{{q}^{N}}$, then $k=1$. Also, $r_{1}=1$ unless $p_{1}(T)M$ is a proper submodule of $M$. Therefore, we have $M$ is
equal to $\mb{F}_{q}[x]/(p_{1}(x))$, where $p_{1}$ is an irreducible
polynomial. This is exactly what it means for $x(=T)$ to act as the
primitive element of the field extension
$\mb{F}_{q}(\sigma)=\mb{F}_{q^{N}}=\mb{F}_{q}^{N}$ with minimal
polynomial $p_{1}(x)$.
\end{proof}

We next define notation to describe the sets to be counted by the general recursion. 

\begin{Def}
Suppose that $A_{1},A_{2},\ldots,A_{k}$ are sets of subspaces of $\mb{F}_{q^{N}}$.
Let $[A_{1},A_{2},\ldots,A_{k}]$ be the set of all $k$-tuples
$(W_{1},W_{2},\ldots,W_{k})$
such that
\begin{alignat*}{2}
W_{i}\in A_{i}&\quad\text{for}&\ 1\leq &\ i\leq k,
\\W_{i}\supseteq W_{i+1}+\sigma W_{i+1}&\quad\text{for}&\quad\ 1\leq &\ i\leq k-1.
\end{alignat*}
\end{Def}
If $A_{i}$ is the set of all subspaces of $\mb{F}_{q^{N}}$ with dimension $d_{i}$, then $A_{i}$ is denoted within the brackets as $d_{i}$. For example, $[3,A_{2}]$ denotes all tuples $(W_{1},W_{2})$ such that $\dim(W_{1})=3$, $W_{2}\in A_{2}$ and $W_{1}\supseteq W_{2}+\sigma W_{2}$.

\begin{Def}
For nonnegative integers $a$, $b$ with $N>a>b$ or $a=b=0$
\begin{align*}
(a,b):=\left\{W\subseteq \mb{F}_{q^{N}}: \dim(W)=a\ \text{and}\ \dim(W\cap \sigma^{-1} W)=b\right\}.
\end{align*}
\end{Def}
For example, $(1,0)$ is the set of all 1-dimensional subspaces and $(2,1)$ is the set of all 2-dimensional subspaces $W$ such that $\dim(W\cap \sigma^{-1} W)=1$.

\begin{Def}
Given sets $[A_{1,1},A_{1,2}],[A_{2,1},A_{2,2}],\ldots,[A_{r,1},A_{r,2}]$ as defined above, let
\begin{align*}
\left\langle[A_{1,1},A_{1,2}],[A_{2,1},A_{2,2}],\ldots,[A_{r,1},A_{r,2}]\right\rangle
\end{align*}
denote the set of $2r$-tuples of subspaces $(W_{1,1},W_{1,2},W_{2,1},W_{2,2},\ldots W_{r,1},W_{r,2})$ such that
\begin{alignat*}{2}
(W_{i,1},W_{i,2})\in[A_{i,1},A_{i,2}]&\quad\text{for}&\quad 1\leq &\ i\leq r,
\\W_{i,2}\supseteq W_{i+1,1}&\quad\text{for}&\quad 1\leq &\ i\leq r-1.
\end{alignat*}
\end{Def}

For example, $\left\langle [3,2],[2,1]\right\rangle$ is the set of all $4$-tuples of subspaces $(W_{1},W_{2},W_{3},W_{4})$ such that
\begin{alignat*}{2}
\dim(W_{1})=3,\ \dim(W_{2})=2,&& &\quad \dim(W_{3})=2,\ \dim(W_{4})=1,
\\W_{1}\supseteq W_{2}+\sigma W_{2},&& &\quad W_{3}\supseteq W_{4}+\sigma W_{4},
\\&& &\quad W_{2}\supseteq W_{3}.
\end{alignat*}

We use the following proposition extensively in constructing the recursion.
\begin{Prop}
\label{expand}
For nonnegative integers $N>a>b$ or $a=b=0$
\begin{align*}
[a,b]&=\bigcup_{i=b}^{\max(a-1,0)}{[(a,i),b]}
\\&=\bigcup_{j=0}^{\max(b-1,0)}{[a,(b,j)]}.
\end{align*}
\end{Prop}

\begin{proof}
Follows from Proposition \ref{preserve} and the definitions of $[\ ,\ ],(\ ,\ )$.
\end{proof}

We next define an ordering on the tuples labelling the sets of subspaces 
\begin{align*}
[(a_{1,1},a_{1,2}),(a_{2,1},a_{2,2}),\ldots,(a_{r,1},a_{r,2})].
\end{align*}
The recursion in Lemma \ref{recursion} will give the cardinality of sets of subspaces so labelled in terms of the cardinality of sets labelled by tuples before it in the ordering. The base case is $[(0,0)]$, containing one element.

\begin{Def}
First, define an ordering on the ordered pairs of the form $(a,b)$ such that $(a_{1},b_{1})\succ(a_{2},b_{2})$ if $a_{1}>a_{2}$ or $a_{1}=a_{2}$ and $b_{1}<b_{2}$.
Next, define an ordering on tuples of the form $[(a_{1,1},a_{1,2}),(a_{2,1},a_{2,2}),\ldots,(a_{r,1},a_{r,2})]$ such that the order is lexicographic in terms of the ordered pairs $(a_{i,1},a_{i,2})$ from left to right.
Finally, define an ordering on the same tuples for $s\geq 0$ such that
\begin{align*} [(a_{1,1},a_{1,2}),(a_{2,1},a_{2,2}),\ldots,(a_{r+s,1},a_{r+s,2})]\succ[(a_{1,1},a_{1,2}),(a_{2,1},a_{2,2}),\ldots,(a_{r,1},a_{r,2})].
\end{align*}

For example, $(3,1)\succ(3,2)\succ(2,0)$ and $[(6,5),(4,2)]\succ[(6,5),(4,3)]\succ[(5,2),(2,0)]$.
\end{Def}

\begin{Lem}
\label{recursion}
Suppose
\begin{align*}
N>a_{1,1}>a_{1,2}\geq a_{2,1}>a_{2,2}\geq\ldots\geq a_{r,1}>a_{r,2}\geq 0=a_{r+1,1}=a_{r+1,2}=\ldots=a_{r+s,1}=a_{r+s,2}
\end{align*}
and after setting
\begin{align*}
a_{0,1}=a_{0,2}=N,\quad a_{r+1,1}=a_{r+1,2}=0,\quad j_{r+1}=k_{r+1}=0,
\end{align*}
that (or else $[(a_{1,1},a_{1,2}),(a_{2,1},a_{2,2}),\ldots,(a_{r,1},a_{r,2})]$ is empty)
\begin{align*}
a_{i-1,1}\geq 2a_{i,1}-a_{i,2}\quad\text{for}\quad 1\leq i\leq r.
\end{align*}

Let
\begin{align*}
&C=\left\{(j_{1},\ldots,j_{r}):\max(a_{i+1,2},2a_{i,2}-a_{i,1})\leq j_{i}\leq \max(a_{i,2}-1,0), 1\leq i\leq r\right\},
\\&D=\left\{(k_{1},\ldots,k_{r}):a_{i,2}\leq k_{i}\leq a_{i,1}-1, 1\leq i\leq r\right\}.
\end{align*}
 
Then
\begin{align*}
&|[(a_{1,1},a_{1,2}),(a_{2,1},a_{2,2}),\ldots,(a_{r+s,1},a_{r+s,2})]|
\\&=|[(a_{1,1},a_{1,2}),(a_{2,1},a_{2,2}),\ldots,(a_{r,1},a_{r,2})]|
\\&=\sum_{(j_{1},\ldots,j_{r})\in C}{|[(a_{1,2},j_{1}),(a_{2,2},j_{2}),\ldots,(a_{r,2},j_{r})]|\prod_{i=1}^{r}{\gauss{a_{i-1,2}-(2a_{i,2}-j_{i})}{a_{i,1}-(2a_{i,2}-j_{i})}}}
\\ &-\sum_{(k_{1},\ldots,k_{r})\in D\backslash (a_{1,2},\ldots,a_{r,2})} {|[(a_{1,1},k_{1}),(a_{2,1},k_{2}),\ldots,(a_{r,1},k_{r})]|\prod_{i=1}^{r}{\gauss{k_{i}-a_{i+1,1}}{a_{i,2}-a_{i+1,1}}}}.
\end{align*}
\end{Lem}

\begin{proof}
We give an example before the general case. Let $r=2$; we compute $|[(3,1),(1,0)]|$ by counting $|\left\langle [3,1],[1,0]\right\rangle|$ in two different ways. Applying Proposition \ref{expand} to the terms on the left within the brackets gives
\begin{align*}
|\left\langle [3,1],[1,0]\right\rangle|=|\left\langle [(3,2),1],[(1,0),0]\right\rangle|+|\left\langle [(3,1),1],[(1,0),0]\right\rangle|.
\end{align*}

Above, if $(W_{1},W_{2},W_{3},W_{4})\in \left\langle [(3,2),1],[(1,0),0]\right\rangle$ then $W_{2}=W_{3}$ and $W_{4}=\left\{0\right\}$. So 
\begin{align*}
|\left\langle [(3,2),1],[(1,0),0]\right\rangle|=|[(3,2),(1,0)]|.
\end{align*} 

Likewise for $(W_{1},W_{2},W_{3},W_{4})\in \left\langle [(3,1),1],[(1,0),0]\right\rangle$ then $W_{2}=W_{3}$ and $W_{4}=\left\{0\right\}$. So 
\begin{align*}
|\left\langle [(3,1),1],[(1,0),0]\right\rangle|=|[(3,1),(1,0)]|,
\end{align*}

and
\begin{align*} 
|\left\langle [3,1],[1,0]\right\rangle|=|[(3,2),(1,0)]|+|[(3,1),(1,0)]|.
\end{align*}

Next, applying Proposition \ref{expand} to the terms on the right within the brackets gives
\begin{align*}
|\left\langle [3,1],[1,0]\right\rangle|=|\left\langle [3,(1,0)],[1,(0,0)]\right\rangle|.
\end{align*}

If $(W_{1},W_{2},W_{3},W_{4})\in \left\langle [3,(1,0)],[1,(0,0)]\right\rangle$, then $W_{3}=W_{2}$, $W_{4}=\left\{0\right\}$ and thus $W_{1}$ is a $3$-dimensional subspace containing the $2$-dimensional space $W_{2}+\sigma W_{2}$. So
\begin{align*}
|\left\langle [3,(1,0)],[1,(0,0)]\right\rangle|=|[(1,0),(0,0)]|\gauss{N-2}{1},
\end{align*}

and therefore
\begin{align*}
|\left\langle [3,(1,0)],[1,(0,0)]\right\rangle|=|[(1,0),(0,0)]|\gauss{N-2}{1}.
\end{align*}

We then have, after rearranging, that
\begin{align*}
|[(3,1),(1,0)]|=|[(1,0),(0,0)]|\gauss{N-2}{1}-|[(3,2),(1,0)]|.
\end{align*} 

Note that 
\begin{align*}
[(3,2),(1,0)],\ [(1,0),(0,0)]
\end{align*}
come before $[(3,1),(1,0)]$ in the ordering on tuples. 

The proof of the Lemma is a generalization of this process. The first equality is clear. The size of
$[(a_{1,1},a_{1,2}),(a_{2,1},a_{2,2}),\ldots,(a_{r,1},a_{r,2})]$ is computed by applying Proposition \ref{expand}
\begin{align*}
&|\left\langle[a_{1,1},a_{1,2}],[a_{2,1},a_{2,2}],\ldots,[a_{r,1},a_{r,2}]\right\rangle|
\\ &=\sum_{(k_{1},\ldots,k_{r})\in D} {|\left\langle[(a_{1,1},k_{1}),a_{1,2}],[(a_{2,1},k_{2}),a_{2,2}],\ldots,[(a_{r,1},k_{r}),a_{r,2}]\right\rangle|}
\\&=\sum_{(k_{1},\ldots,k_{r})\in D} {|[(a_{1,1},k_{1}),(a_{2,1},k_{2}),\ldots,(a_{r,1},k_{r})]|\prod_{i=1}^{r}{\gauss{k_{i}-a_{i+1,1}}{a_{i,2}-a_{i+1,1}}}}\label{right}\tag{R}.
\intertext{Expanding in the other way, we get}
&|\left\langle[a_{1,1},a_{1,2}],[a_{2,1},a_{2,2}],\ldots,[a_{r,1},a_{r,2}]\right\rangle|
\\&=\sum_{(j_{1},\ldots,j_{r})\in C}{|\left\langle[a_{1,1},(a_{1,2},j_{1})],[a_{2,1},(a_{2,2},j_{2})],\ldots,[a_{r,1},(a_{r,2},j_{r})]\right\rangle|}
\\&=\sum_{(j_{1},\ldots,j_{r})\in C}{|[(a_{1,2},j_{1}),(a_{2,2},j_{2}),\ldots,(a_{r,2},j_{r})]|\prod_{i=1}^{r}{\gauss{a_{i-1,2}-(2a_{i,2}-j_{i})}{a_{i,1}-(2a_{i,2}-j_{i})}}}\label{left}\tag{L}.
\end{align*}
Subtracting from \eqref{right} and \eqref{left} the quantity
\begin{align*}
|[(a_{1,1},a_{1,2}),(a_{2,1},a_{2,2}),\ldots,(a_{r,1},a_{r,2})]|
\end{align*}
produces the stated result of the Lemma.
\end{proof}

Finally, we relate sets of the form $[(a_{1,1},a_{1,2}),(a_{2,1},a_{2,2}),\ldots,(a_{r,1},a_{r,2})]$ to $\sigma$-splitting subspaces.

\begin{Prop}
\label{splittingsubspaces}
Let $\mb{F}_{q^{N}}=\mb{F}_{q^{mn}}$. Then
\begin{align*}
&[((n-1)m,(n-2)m),((n-2)m,(n-3)m),\ldots,(2m,m),(m,0)]
\\&=\left\{(\bigoplus_{i=0}^{n-2}{\sigma^{i}W},\ \bigoplus_{i=0}^{n-3}{\sigma^{i}W},\ \ldots,\ W\oplus\sigma W,\ W)\ :\  \bigoplus_{i=0}^{n-1}{\sigma^{i}W}=\mb{F}_{q^{mn}}\right\}.
\end{align*}

In particular $|[((n-1)m,(n-2)m),((n-2)m,(n-3)m),\ldots,(2m,m),(m,0)]|$ is the number of $\sigma$-splitting subspaces.
\end{Prop}

\begin{proof}
If W is a $\sigma$-splitting subspace, then
\begin{align*}
(\bigoplus_{i=0}^{n-2}&{\sigma^{i}W},\bigoplus_{i=0}^{n-3}{\sigma^{i}W},\ldots,W\oplus\sigma W,W)
\\&\in[((n-1)m,(n-2)m),((n-2)m,(n-3)m),\ldots,(2m,m),(m,0)].
\end{align*}

On the other hand, suppose that
\begin{align*}
(W_{n-1},\ldots,W_{1})\in [((n-1)m,(n-2)m),((n-2)m,(n-3)m),\ldots,(2m,m),(m,0)] 
\end{align*}
Then for $1\leq k\leq n-2$
\begin{align*}
\dim(W_{k+1})&=(k+1)m
\\&=2km-(k-1)m
\\&=\dim(W_{k}+\sigma W_{k}).
\intertext{So $W_{k+1}=W_{k}+\sigma W_{k}$ for $1\leq k\leq n-2$. Also, $W_{2}=W_{1}\oplus\sigma W_{1}$ as $W_{1}\cap\sigma W_{1}=\left\{0\right\}$.} 
\intertext{Suppose that $W_{k}=\bigoplus_{i=0}^{k-1}{\sigma^{i}W_{1}}$. Then, since $\dim(W_{k+1})=\dim(W_{k}+\sigma W_{k})=(k+1)m$, we obtain}
W_{k+1}&=W_{k}+\sigma W_{k}
\\&=W_{1}+\sigma W_{1}+\cdots+\sigma^{k}W_{1}
\\&=\bigoplus_{i=0}^{k}{\sigma^{i}W_{1}}.
\end{align*}

When $k=n-1$, we have that $W_{n-1}+\sigma W_{n-1}=\bigoplus_{i=0}^{n-1}{\sigma^{i}W_{1}}$, since $W_{n-1}+\sigma W_{n-1}=\mb{F}_{q^{mn}}$ is $mn$-dimensional. So $W_{1}$ is indeed a $\sigma$-splitting subspace.
\end{proof}

\begin{Cor}
The number of $\sigma$-splitting subspaces in $\mb{F}_{q^{mn}}$ over $F_{q}$ is independent of choice of primitive element $\sigma$.
\end{Cor}

\begin{proof}
Neither the base case $|[(0,0)]|$ nor Lemma \ref{recursion} depends on the $\sigma$ chosen.
\end{proof}

\begin{Rem}
More generally, given an arbitrary invertible linear operator $T$ on $\mb{F}_{q^{mn}}$ over $\mb{F}_{q}$, we might consider how many ``$T$-splitting'' subspaces exist; that is, the number of $m$-dimensional subspaces $W$ such that 
\begin{align*}
W\oplus TW\oplus\cdots\oplus T^{n-1}W=\mb{F}_{q^{mn}}.
\end{align*}

We may then redefine $(\ ,\ ),[\ ,\ ],\left\langle\ ,\ \right\rangle$ by replacing the expressions $W+\sigma W$ with $W+TW$ and $W\cap\sigma^{-1}W$ with $W\cap T^{-1}W$.

Recall from Proposition \ref{preserve} and Lemma \ref{recursion} that when $T(v)=\sigma v$, the nonzero numbers $|[(a_{1,1},a_{1,2}),(a_{2,1},a_{2,2}),\ldots,(a_{r,1},a_{r,2})]|$ can be computed from the base case $|[0,0]|=1$. But if $T$ is any invertible linear operator, there may exist nonempty sets of the form $[(a_{1},a_{1}),(a_{2},a_{2}),\ldots,(a_{r},a_{r})]$ where $a_{r}\neq 0$. In fact, such sets cannot be computed recursively. For example
\begin{align*}
&|\left\langle[4,4],[2,2]\right\rangle|
\\&=|\left\langle[(4,4),4],[(2,2),2]\right\rangle|=|[(4,4),(2,2)]|
\\&=|\left\langle[4,(4,4)],[2,(2,2)]\right\rangle|=|[(4,4),(2,2)]|.
\end{align*}
We may still apply Lemma \ref{recursion} in the case of general $T$, however, with the cardinalities of these sets as additional base cases.
\end{Rem}

\section{Solution to the Recursion}
The next two lemmas are special cases of the following $q$-Chu-Vandermonde  identity for $N$ a nonnegative integer \cite[p.~354]{BHS}.
\begin{align*}
\ _2\phi_1
\left(
\begin{matrix}q^{-N},&a;&cq^{N}/a\\
&c&
\end{matrix}
\right)&:=\sum_{m=0}^{N}{\frac{(q^{-N};q)_{m}(a;q)_{m}}{(q;q)_{m}(c;q)_{m}}\left(\frac{cq^{N}}{a}\right)^{m}}
\\&=\frac{(c/a;q)_N}{(c;q)_N}.
\end{align*}

\begin{Lem} 
\label{firstidentity}
If $C\le B-1\le D-1\le A-1$ are non-negative integers, then

\begin{align*}
&\sum_{s=C}^{B-1} \gauss{A-B-1}{B-s-1}\gauss{B}{s} 
\gauss{s}{C}\gauss{A-(2B-s)}{D-(2B-s)} q^{(B-s)(B-s-1)}\\
&=
\frac{[B]_q}{[D-C]_q}
\gauss{B-1}{C}\gauss{A-B-1}{D-B-1}\gauss{D-C}{B-C}.
\end{align*}
\end{Lem}

\begin{Lem} 
\label{secondidentity}
If $C\le D\le B-1\le A-1$ are non-negative integers, then
\begin{align*}
&\sum_{s=D}^{B-1} \gauss{A-B-1}{B-s-1}\gauss{B}{s} 
\gauss{s}{C}\gauss{s-C}{D-C} q^{(B-s)(B-s-1)}
\\&=\frac{[B]_q}{[A-D]_q}
\gauss{B-1}{C}\gauss{B-C-1}{D-C}\gauss{A-D}{B-D}.
\end{align*}
\end{Lem}

We now give our main theorem.
\begin{Thm}\label{alpha}
Suppose that 
\begin{align*}
N>a_{1,1}>a_{1,2}\ge a_{2,1}>a_{2,2}\ge \cdots \ge a_{r,1}>a_{r,2}\ge 0,
\\
a_{0,1}=a_{0,2}=N, \quad a_{r+1,1}=a_{r+1,2}=0.
\end{align*} 

Then
\begin{align}
|[(a_{1,1},a_{1,2}),\ldots, (a_{r,1},a_{r,2})]|=
\frac{\gauss{N}{1}}{\gauss{a_{1,1}}{1}} \frac{\prod_{i=0}^{r-1} 
\gauss{a_{i,1}-a_{i+1,1}-1}{a_{i+1,1}-a_{i+1,2}-1}
\gauss{a_{i+1,1}}{a_{i+1,2}}
\gauss{a_{i+1,2}}{a_{i+2,1}}}
{\prod_{i=1}^{r-1} \gauss{a_{i,1}-1}{a_{i+1,1}-1}}q^E\label{formula},
\end{align}

where
\begin{align*}
E=\sum_{i=1}^r (a_{i,1}-a_{i,2})(a_{i,1}-a_{i,2}-1).
\end{align*}
\end{Thm}

\begin{Cor}[Splitting Subspace Conjecture]
\label{grssc}
We have, when $N\geq mn$, the equality
\begin{align*}
|[((n-1)m,(n-2)m),\ldots, (2m,m), (m,0)]|&=\frac{\gauss{N}{1}}{\gauss{m}{1}}\gauss{N-mn+m-1}{m-1}q^{m(m-1)(n-1)}
\end{align*}
In particular, when $N=mn$,
\begin{align*}
|[((n-1)m,(n-2)m),\ldots, (2m,m), (m,0)]|
=\frac{\gauss{mn}{1}}{\gauss{m}{1}} q^{m(m-1)(n-1)}.
\end{align*}
\end{Cor}

\begin{proof}
From plugging into \eqref{formula}
\begin{align*}
&[((n-1)m,(n-2)m),((n-2)m,(n-3)m),\ldots,(2m,m),(m,0)]
\\&=
\frac{\gauss{N}{1}}{\gauss{(n-1)m}{1}}\frac{\left(\gauss{N-(n-1)m-1}{m-1}\gauss{m-1}{m-1}\cdots \gauss{m-1}{m-1}\right)\left(\gauss{(n-1)m}{(n-2)m}\cdots\gauss{m}{0}\right)}{\gauss{(n-1)m-1}{(n-2)m-1}\cdots\gauss{m-1}{0}}q^{\sum_{i=1}^{n-1}{m(m-1)}}
\\&=
\frac{\gauss{N}{1}}{\gauss{(n-1)m}{1}}\gauss{N-(n-1)m-1}{m-1}\frac{\gauss{(n-1)m}{(n-2)m}\cdots\gauss{2m}{m}}{\gauss{(n-1)m-1}{(n-2)m-1}\cdots\gauss{m-1}{0}}q^{m(m-1)(n-1)}.
%\end{align*}
%Since for each $1\leq k\leq n-2$, $\frac{\gauss{(n-k)m}{(n-k-1)m}}{\gauss{(n-k)m-1}{(n-k-1)m-1}}=\frac{\frac{[(n-k)m]!_{q}}{[m]!_{q}[(n-k-1)m]!_{q}}}{\frac{[(n-k)m-1]!_{q}}{[m]!_{q}[(n-k-1)m-1]!_{q}}}=\frac{1-q^{(n-k)m}}{1-q^{(n-k-1)m}}$, this reduces to
%\begin{align*}
\\&=\frac{\gauss{N}{1}}{\gauss{(n-1)m}{1}}\gauss{N-(n-1)m-1}{m-1}\frac{1-q^{(n-1)m}}{1-q^{m}}q^{m(m-1)(n-1)}
\\&=\frac{\gauss{N}{1}}{\gauss{m}{1}}\gauss{N-mn+m-1}{m-1}q^{m(m-1)(n-1)}.
\end{align*}
\end{proof}

\begin{proof}[Proof of Theorem \ref{alpha}]
We verify that \eqref{formula} satisfies the 
recursion in Lemma \ref{recursion}.

Recall that
\begin{align*}
\ref*{left}&=\sum_{(j_{1},\ldots,j_{r})\in C}{|[(a_{1,2},j_{1}),(a_{2,2},j_{2}),\ldots,(a_{r,2},j_{r})]|\prod_{i=1}^{r}{\gauss{a_{i-1,2}-(2a_{i,2}-j_{i})}{a_{i,1}-(2a_{i,2}-j_{i})}}},
\\\ref*{right}&=\sum_{(k_{1},\ldots,k_{r})\in D} {|[(a_{1,1},k_{1}),(a_{2,1},k_{2}),\ldots,(a_{r,1},k_{r})]|\prod_{i=1}^{r}{\gauss{k_{i}-a_{i+1,1}}{a_{i,2}-a_{i+1,1}}}}.
\end{align*}

We first check equality when $a_{r,2}\neq 0$ so that the expressions obtained for \eqref{left} and \eqref{right} using \eqref{formula} do not contain negative $q$-binomials.

Substituting \eqref{formula} and applying Lemma \ref{firstidentity} to the resulting independent sums in \eqref{left} gives
\begin{align*}
\ref*{left}&=
\frac{\gauss{N}{1}}{\gauss{a_{1,2}}{1}}
\frac{\prod_{i=1}^{r}{\sum_{j_{i}}{\gauss{a_{i-1,2}-a_{i,2}-1}{a_{i,2}-j_{i}-1}
\gauss{a_{i,2}}{j_{i}}
\gauss{j_{i}}{a_{i+1,2}}
\gauss{a_{i-1,2}-(2a_{i,2}-j_{i})}{a_{i,1}-(2a_{i,2}-j_{i})}q^{(a_{i,2}-j_{i})(a_{i,2}-j_{i}-1)}}}}
{\prod_{i=1}^{r-1} \gauss{a_{i,2}-1}{a_{i+1,2}-1}}
\\&=
\frac{\gauss{N}{1}}{\gauss{a_{1,2}}{1}}
\prod_{i=1}^r \frac{ [a_{i,2}]_q}{[a_{i,1}-a_{i+1,2}]_q}
\gauss{a_{i,2}-1}{a_{i+1,2}}
\gauss{a_{i-1,2}-a_{i,2}-1}{a_{i,1}-a_{i,2}-1}
\gauss{a_{i,1}-a_{i+1,2}}{a_{i,2}-a_{i+1,2}}
\prod_{i=1}^{r-1} \gauss{a_{i,2}-1}{a_{i+1,2}-1}^{-1}.
\end{align*}

Substituting \eqref{formula} and applying Lemma \ref{secondidentity} to the resulting independent sums in \eqref{right} gives
\begin{align*}
\ref*{right}&=
\frac{\gauss{N}{1}}{\gauss{a_{1,1}}{1}}
\frac{\prod_{i=1}^{r}{\sum_{k_{i}}{\gauss{a_{i-1,1}-a_{i,1}-1}{a_{i,1}-k_{i}-1}
\gauss{a_{i,1}}{k_{i}}
\gauss{k_{i}}{a_{i+1,1}}
\gauss{k_{i}-a_{i+1,1}}{a_{i,2}-a_{i+1,1}}q^{(a_{i,1}-k_{i})(a_{i,1}-k_{i}-1)}}}}
{\prod_{i=1}^{r-1} \gauss{a_{i,2}-1}{a_{i+1,2}-1}}
\\&=
\frac{\gauss{N}{1}}{\gauss{a_{1,1}}{1}}
\prod_{i=1}^r \frac{ [a_{i,1}]_q}{[a_{i-1,1}-a_{i,2}]_q}
\gauss{a_{i,1}-1}{a_{i+1,1}}
\gauss{a_{i,1}-a_{i+1,1}-1}{a_{i,2}-a_{i+1,1}}
\gauss{a_{i-1,1}-a_{i,2}}{a_{i,1}-a_{i,2}}
\prod_{i=1}^{r-1} \gauss{a_{i,1}-1}{a_{i+1,1}-1}^{-1}.
\end{align*}

After simplification (see Appendix \ref{prooflr})
\begin{align}
\ref*{left}=\ref*{right}=
\frac{[N]_{q} [N-a_{1,2}-1]!_{q}}{[N-a_{1,1}]!_{q} [a_{r,2}]!_{q}}
\prod_{i=1}^r 
\frac{[a_{i,1}-a_{i+1,2}-1]!_{q}}{[a_{i,1}-a_{i,2}-1]!_{q} [a_{i,1}-a_{i,2}]!_{q}}
\prod_{i=2}^r
\frac{1}{[a_{i-1,2}-a_{i,1}]!_{q}}.\label{identity}
\end{align}

Finally, we deal with the case $a_{r,2}=0$, when the expression obtained by directly applying \eqref{formula} to \eqref{left} may contain negative q-binomials (\eqref{right} is unaffected). Suppose $r>1$. By definition, we know that 
\begin{align*}
|\left\langle[a_{1,1},a_{1,2}],[a_{2,1},a_{2,2}],\ldots,[a_{r,1},0]\right\rangle|=|\left\langle[a_{1,1},a_{1,2}],[a_{2,1},a_{2,2}],\ldots,[a_{r-1,1},a_{r-1,2}]\right\rangle|\gauss{a_{r-1,2}}{a_{r,1}}.
\end{align*}
This means that
\begin{align*}
\ref*{left}&=\sum_{(j_{1},\ldots,j_{r})\in C}{|[(a_{1,2},j_{1}),(a_{2,2},j_{2}),\ldots,(a_{r,2},j_{r})]|\prod_{i=1}^{r}{\gauss{a_{i-1,2}-(2a_{i,2}-j_{i})}{a_{i,1}-(2a_{i,2}-j_{i})}}}
\\&=
|\left\langle[a_{1,1},a_{1,2}],,\ldots,[a_{r-1,1},a_{r-1,2}]\right\rangle|\gauss{a_{r-1,2}}{a_{r,1}}.
\end{align*}
Since $a_{r-1,2}\geq a_{r,1}>0$, we may apply our previous result to obtain 
\begin{align*}
&|\left\langle[a_{1,1},a_{1,2}],\ldots,[a_{r-1,1},a_{r-1,2}]\right\rangle|\gauss{a_{r-1,2}}{a_{r,1}}
\\&=
\gauss{a_{r-1,2}}{a_{r,1}}\frac{[N]_{q} [N-a_{1,2}-1]!_q}{[N-a_{1,1}]!_q [a_{r-1,2}]!_q}
\prod_{i=1}^{r-1} 
\frac{[a_{i,1}-a_{i+1,2}-1]!_{q}}{[a_{i,1}-a_{i,2}-1]!_{q} [a_{i,1}-a_{i,2}]!_{q}}
\prod_{i=2}^{r-1}
\frac{1}{[a_{i-1,2}-a_{i,1}]!_{q}}.
\end{align*}
We wish to show that this is equal to 
\begin{align*}
\frac{[N]_{q} [N-a_{1,2}-1]!_{q}}{[N-a_{1,1}]!_{q} [a_{r,2}]!_{q}}
\prod_{i=1}^r 
\frac{[a_{i,1}-a_{i+1,2}-1]!_{q}}{[a_{i,1}-a_{i,2}-1]!_{q} [a_{i,1}-a_{i,2}]!_{q}}
\prod_{i=2}^r
\frac{1}{[a_{i-1,2}-a_{i,1}]!_{q}}
\end{align*}
when $a_{r,2}=0$. 
Take the quotient to find
\begin{align*}
&\frac{\frac{[N]_{q} [N-a_{1,2}-1]!_{q}}{[N-a_{1,1}]!_{q} [a_{r,2}]!_{q}}
\prod_{i=1}^r 
\frac{[a_{i,1}-a_{i+1,2}-1]!_{q}}{[a_{i,1}-a_{i,2}-1]!_{q} [a_{i,1}-a_{i,2}]!_{q}}
\prod_{i=2}^r
\frac{1}{[a_{i-1,2}-a_{i,1}]!_{q}}}{\gauss{a_{r-1,2}}{a_{r,1}}\frac{[N]_{q} [N-a_{1,2}-1]!_{q}}{[N-a_{1,1}]!_{q} [a_{r-1,2}]!_{q}}
\prod_{i=1}^{r-1} 
\frac{[a_{i,1}-a_{i+1,2}-1]!_{q}}{[a_{i,1}-a_{i,2}-1]!_{q} [a_{i,1}-a_{i,2}]!_{q}}
\prod_{i=2}^{r-1}
\frac{1}{[a_{i-1,2}-a_{i,1}]!_{q}}}
%\\&=
%\frac{[a_{r-1,2}]!_{q}\frac{[a_{r,1}-a_{r+1,2}-1]!_{q}}{[a_{r,1}-a_{r,2}-1]!_{q} [a_{r,1}-a_{r,2}]!_{q}}
%\frac{1}{[a_{r-1,2}-a_{r,1}]!_{q}}}{\gauss{a_{r-1,2}}{a_{r,1}}[a_{r,2}]!_{q}}
\\&=
\frac{[a_{r-1,2}]!_{q}\frac{[a_{r,1}-1]!_{q}}{[a_{r,1}-1]!_{q} [a_{r,1}]!_{q}}
\frac{1}{[a_{r-1,2}-a_{r,1}]!_{q}}}{\gauss{a_{r-1,2}}{a_{r,1}}}
%\\&=
%\frac{[a_{r-1,2}]!_{q}\frac{1}{[a_{r,1}]!_{q}}
%\frac{1}{[a_{r-1,2}-a_{r,1}]!_{q}}}{\gauss{a_{r-1,2}}{a_{r,1}}}
\\&=1,
\end{align*}
as desired. Therefore, when $a_{r,2}=0$, the equality
\begin{align*}
\ref*{left}&=\sum_{(j_{1},\ldots,j_{r})\in C}{|[[a_{1,2},j_{1}),[a_{2,2},j_{2}),\ldots,[a_{r,2},j_{r})]|\prod_{i=1}^{r}{\gauss{a_{i-1,2}-(2a_{i,2}-j_{i})}{a_{i,1}-(2a_{i,2}-j_{i})}}}
\\&=\frac{[N]_{q} [N-a_{1,2}-1]!_{q}}{[N-a_{1,1}]!_{q} [a_{r,2}]!_{q}}
\prod_{i=1}^r 
\frac{[a_{i,1}-a_{i+1,2}-1]!_{q}}{[a_{i,1}-a_{i,2}-1]!_{q} [a_{i,1}-a_{i,2}]!_{q}}
\prod_{i=2}^r
\frac{1}{[a_{i-1,2}-a_{i,1}]!_{q}}
\\&=
\ref*{right}
\end{align*}
still holds.

Finally, suppose $r=1$. Then,
\begin{align*}
\ref*{left}&=|[(0,0)]|\gauss{N}{a_{1,1}}=\gauss{N}{a_{1,1}}.
\end{align*}
If we plug in $\langle [a_{1,1},0]\rangle$ into \eqref{identity}, then we get $\gauss{N}{a_{1,1}}$, as desired. 
\end{proof}

\begin{Cor}
\label{RLequality}
The numbers
\begin{align*}
|\left\langle[a_{1,1},a_{1,2}], [a_{21},a_{22}], \ldots, [a_{r,1},a_{r,2}]\right\rangle|
\end{align*}
are given by
\begin{align*}
\ref{left}=\ref{right}=
\frac{[N]_{q} [N-a_{1,2}-1]!_{q}}{[N-a_{1,1}]!_{q} [a_{r,2}]!_{q}}
\prod_{i=1}^r 
\frac{[a_{i,1}-a_{i+1,2}-1]!_{q}}{[a_{i,1}-a_{i,2}-1]!_{q} [a_{i,1}-a_{i,2}]!_{q}}
\prod_{i=2}^r
\frac{1}{[a_{i-1,2}-a_{i,1}]!_{q}}.
\end{align*}

\end{Cor}

\section{Special Case: (k,k-1)}
Note that when $r=1$ and $a_{1,1}=k,a_{1,2}=k-1$, with $k\leq N-1$, the formula \eqref{formula} gives
\begin{align*}
|(k,k-1)|=\gauss{N}{1},
\end{align*}
a number independent of $k$.

\begin{Prop}
There is a bijection between sets of the form $(k_{1},k_{1}-1)$ and $(k_{2},k_{2}-1)$ when $k_{1},k_{2}\leq N-1$.
\end{Prop}

\begin{proof}
It suffices to show that there exists a bijection between $(k,k-1)$ and $(k-1,k-2)$ for $2\leq k\leq N-1$. Define
\begin{align*}
\phi:(k-1,k-2)&\rightarrow (k,k-1)\\ W&\mapsto W+\sigma W.
\end{align*}

The map $\phi$ is well defined:
\begin{align*}
\dim(W+\sigma W)&=k,
\\ \dim((W+\sigma W)\cap(\sigma^{-1}W+W))&=k-1.
\end{align*}

\noindent
The second equality follows from the fact that $(W+\sigma W)\cap(\sigma^{-1}W+W)$ contains $W$ and has 
dimension \emph{strictly} less than $k$ by Proposition~\ref{preserve}.
%The second equality follows from $W\subseteq(W+\sigma W)\cap(\sigma^{-1}W+W)$ and the fact that
%\begin{align*}
%k-1=\dim(W)\leq\dim((W+\sigma W)\cap(\sigma^{-1}W+W))<k, 
%\end{align*}
%where the strict inequality follows from Proposition~\ref{preserve}.

Next, $\phi$ is injective:
if $W_{1},\ W_{2}\in(k-1,k-2)$ and $W_{1}+\sigma W_{1}=W_{2}+\sigma W_{2}=W'\in(k,k-1)$, then $W'\cap\sigma^{-1}W'=W_{1}=W_{2}$.

Finally, $\phi$ is surjective:
if $W'\in(k,k-1)$, then $W'\cap\sigma^{-1}W'\in (k-1,k-2)$ since $(W'\cap\sigma^{-1}W')+\sigma(W'\cap\sigma^{-1}W')\subseteq W'$; in fact $(W'\cap\sigma^{-1}W')+\sigma(W'\cap\sigma^{-1}W')=W'$.
\end{proof}

\section{A $q=1$ analogue}
We might ask what \eqref{formula} counts when $q=1$; indeed the situation translates from enumerating subspaces of vector spaces to enumerating subsets of sets.

Instead of subspaces of $\mb{F}_{{q}^{N}}$, we consider subsets of $\{1,\ldots,N\}$. Rather than multiplying by the element $\sigma$, we let $\sigma=(12\cdots N)$ cyclically permute the elements of $\{1,\ldots,N\}$ so that $\sigma$ preserves no proper subset, in analogy with Proposition \ref{preserve}; in fact, any permutation of $\{1,\ldots,N\}$ preserving no proper subset is cyclic. When $N=mn$, the number of $m$-element subsets $W$ of $\{1,\ldots,N\}$ such that $\bigcup_{i=0}^{n-1}{\sigma^{i}W}=\{1,\ldots,N\}$ is clearly $n$; this is true in a more general setting. We retain the $[\ ,\ ],(\ ,\ ),<\ ,\ >$ notation with the definitions restated in the context of subsets of $\{1,\ldots,N\}$ below.

\begin{Def}
Suppose $A_{1},A_{2},\ldots,A_{k}$ are sets of subsets of $\{1,\ldots,N\}$.
Let $[A_{1},A_{2},\ldots,A_{k}]$ be the set of all $k$-tuples $(W_{1},W_{2},\ldots,W_{k})$ such that
\begin{alignat*}{2}
W_{i}\in A_{i}&&\quad\text{for}&\quad 1\leq i\leq k,
\\W_{i}\supseteq W_{i+1}\cup\sigma W_{i+1}&&\quad\text{for}&\quad 1\leq i\leq k-1.
\end{alignat*} 
\end{Def}
If $A_{i}$ is the set of all subsets of $\{1,\ldots,N\}$ with cardinality $d_{i}$, then $A_{i}$ is denoted within the brackets as $d_{i}$.

\begin{Def}
For nonnegative integers $a$, $b$ with $N>a>b$ or $a=b=0$
\begin{align*}
(a,b):=\left\{W\subseteq \{1,\ldots,N\}: |W|=a,|W\cap \sigma^{-1} W|=b\right\}.
\end{align*}
\end{Def}

\begin{Def}
Given sets $[A_{1,1},A_{1,2}],[A_{2,1},A_{2,2}],\ldots,[A_{r,1},A_{r,2}]$ as defined above, let
\begin{align*}
\left\langle[A_{1,1},A_{1,2}],[A_{2,1},A_{2,2}],\ldots,[A_{r,1},A_{r,2}]\right\rangle
\end{align*}
denote the set of $2r$-tuples of subsets $(W_{1,1},W_{1,2},W_{2,1},W_{2,2},\ldots W_{r,1},W_{r,2})$ such that
\begin{alignat*}{2}
(W_{i,1},W_{i,2})\in[A_{i,1},A_{i,2}]&&\quad\text{for}&\quad 1\leq i\leq r,
\\W_{i,2}\supseteq W_{i+1,1}&&\quad\text{for}&\quad 1\leq i\leq r-1.
\end{alignat*}
\end{Def}

Lemma \ref{recursion} and our formulas in Theorem \ref{alpha} and Corollaries \ref{grssc} and \ref{RLequality} are still valid here by setting $q=1$; the $q$-binomials counting ways to extend subspaces become binomial terms counting ways to enlarge subsets. However, we can directly count some special cases. 

An appropriate adaptation of Proposition \ref{splittingsubspaces} gives
\begin{align*}
&[((n-1)m,(n-2)m),((n-2)m,(n-3)m),\ldots,(2m,m),(m,0)]
\\&=\left\{(\bigcup_{i=0}^{n-2}{\sigma^{i}W},\ \bigcup_{i=0}^{n-3}{\sigma^{i}W},\ \ldots,\ W\cup\sigma W,\ W)\ :\  |\bigcup_{i=0}^{n-1}{\sigma^{i}W}|=mn\ \text{and}\ |W|=m\right\}.
\end{align*}

Counting the number of ordered pairs $(W,k)$ where $W$ satisfies the conditions in the set above and $k$ is an element of $W$ in two different ways, one
by fixing $k$ first and the other by fixing $W$ first, yields
$|[((n-1)m,(n-2)m),\ldots,(m,0)]|=\frac{N}{m}\binom{N-mn+m-1}{m-1}$,
which is the same as the formula from Corollary \ref{grssc} when we set $q=1$.

We may also count the elements of $(m,k)$ directly
by counting the number of ordered pairs $(W,a)$ where $W\in (m,k)$
and $a\notin W$ in two ways, one by fixing $W$ first
and the other by fixing $a$ first. This yields
$\frac{N}{N-m}\binom{N-m}{m-k}\binom{m-1}{k}=\frac{N}{m}\binom{N-m-1}{m-k-1}\binom{m}{k}$. For $q$ a power of a prime, \eqref{formula} yields 
\begin{align*}
|(m,k)|=\frac{\gauss{N}{1}}{\gauss{m}{1}}\gauss{N-m-1}{m-k-1}\gauss{m}{k}q^{(m-k)(m-k-1)}.
\intertext{This gives the same expression when $q\rightarrow 1$ as obtained above.} 
\end{align*}

\appendix
\section{Proof of Theorem \ref{alpha}, \ref{left}=\ref{right}}
\label{prooflr}
\begin{proof}
\begin{align*}
\ref{left}&=\frac{\gauss{N}{1}}{\gauss{a_{1,2}}{1}}
\prod_{i=1}^r \frac{ [a_{i,2}]_q}{[a_{i,1}-a_{i+1,2}]_q}
\gauss{a_{i,2}-1}{a_{i+1,2}}
\gauss{a_{i-1,2}-a_{i,2}-1}{a_{i,1}-a_{i,2}-1}
\gauss{a_{i,1}-a_{i+1,2}}{a_{i,2}-a_{i+1,2}}
\prod_{i=1}^{r-1} {\gauss{a_{i,2}-1}{a_{i+1,2}-1}^{-1}}
\\&=\frac{\frac{[N]!_{q}}{[N-1]!_{q}}}{\frac{[a_{1,2}]!_{q}}{[a_{1,2}-1]!_{q}}}\frac{[a_{r,1}-1]!_{q}}{[a_{r,2}-1]!_{q}}\gauss{a_{r-1,2}-a_{r,2}-1}{a_{r,1}-a_{r,2}-1}\frac{1}{[a_{r,1}-a_{r,2}]!_{q}}
\\&\prod_{i=1}^{r-1}{\frac{[a_{i,2}]!_{q}}{[a_{i,2}-1]!_{q}}\frac{[a_{i,1}-a_{i+1,2}-1]!_{q}}{[a_{i+1,2}]!_{q}[a_{i,2}-a_{i+1,2}-1]!_{q}}\frac{[a_{i-1,2}-a_{i,2}-1]!_{q}}{[a_{i,1}-a_{i,2}-1]!_{q}[a_{i-1,2}-a_{i,1}]!_{q}}\frac{[a_{i+1,2}-1]!_{q}}{[a_{i,1}-a_{i,2}]!_{q}}}
\\&=\frac{\frac{[N]!_{q}}{[N-1]!_{q}}}{\frac{[a_{1,2}]!_{q}}{[a_{1,2}-1]!_{q}}}\frac{[a_{r,1}-1]!_{q}}{[a_{r,2}-1]!_{q}}\frac{[a_{r-1,2}-a_{r,2}-1]!_{q}}{[a_{r,1}-a_{r,2}-1]!_{q}[a_{r-1,2}-a_{r,1}]!_{q}}\frac{1}{[a_{r,1}-a_{r,2}]!_{q}}
\\&\frac{[a_{1,2}]!_{q}}{[a_{r,2}]!_{q}}\frac{[a_{r,2}-1]!_{q}}{[a_{1,2}-1]!_{q}}\frac{[N-a_{1,2}-1]!_{q}}{[a_{r-1,2}-a_{r,2}-1]!_{q}}\prod_{i=1}^{r-1}{\frac{[a_{i,1}-a_{i+1,2}-1]!_{q}}{[a_{i,1}-a_{i,2}-1]!_{q}[a_{i-1,2}-a_{i,1}]!_{q}[a_{i,1}-a_{i,2}]!_{q}}}
\\&=\frac{[N]_{q} [N-a_{1,2}-1]!_{q}}{[N-a_{1,1}]!_{q} [a_{r,2}]!_{q}}
\prod_{i=1}^r 
\frac{[a_{i,1}-a_{i+1,2}-1]!_{q}}{[a_{i,1}-a_{i,2}-1]!_{q} [a_{i,1}-a_{i,2}]!_{q}}
\prod_{i=2}^r
\frac{1}{[a_{i-1,2}-a_{i,1}]!_{q}}.
\end{align*}

$\nonumber$

\begin{align*}
\ref{right}&=\frac{\gauss{N}{1}}{\gauss{a_{1,1}}{1}}
\prod_{i=1}^r \frac{ [a_{i,1}]_q}{[a_{i-1,1}-a_{i,2}]_q}
\gauss{a_{i,1}-1}{a_{i+1,1}}
\gauss{a_{i,1}-a_{i+1,1}-1}{a_{i,2}-a_{i+1,1}}
\gauss{a_{i-1,1}-a_{i,2}}{a_{i,1}-a_{i,2}}
\prod_{i=1}^{r-1} \gauss{a_{i,1}-1}{a_{i+1,1}-1}^{-1}
%\\&=\frac{\frac{[N]!_{q}}{[N-1]!_{q}[1]!_{q}}}{\frac{[a_{1,1}]!_{q}}{[a_{1,1}-1]!_{q}[1]!_{q}}}\frac{\frac{[a_{r,1}]!_{q}}{[a_{r,1}-1]!_{q}[1]!_{q}}}{\frac{[a_{r-1,1}-a_{r,2}]!_{q}}{[a_{r-1,1}-a_{r,2}-1]!_{q}[1]!_{q}}}\frac{[a_{r,1}-1]!_{q}}{[a_{r,2}]!_{q}[a_{r,1}-a_{r,2}-1]!_{q}}\frac{[a_{r-1,1}-a_{r,2}]!_{q}}{[a_{r,1}-a_{r,2}]!_{q}[a_{r-1,1}-a_{r,1}]!_{q}}
%\\&\prod_{i=1}^{r-1}{\frac{\frac{[a_{i,1}]!_{q}}{[a_{i,1}-1]!_{q}[1]!_{q}}}{\frac{[a_{i-1,1}-a_{i,2}]!_{q}}{[a_{i-1,1}-a_{i,2}-1]!_{q}[1]!_{q}}}\frac{[a_{i,1}-1]!_{q}}{[a_{i+1,1}]!_{q}[a_{i,1}-a_{i+1,1}-1]!_{q}}\frac{[a_{i,1}-a_{i+1,1}-1]!_{q}}{[a_{i,2}-a_{i+1,1}]!_{q}[a_{i,1}-a_{i,2}-1]!_{q}}\frac{\gauss{a_{i-1,1}-a_{i,2}}{a_{i,1}-a_{i,2}}}{\gauss{a_{i,1}-1}{a_{i+1,1}-1}}}
\\&=\frac{\frac{[N]!_{q}}{[N-1]!_{q}}}{\frac{[a_{1,1}]!_{q}}{[a_{1,1}-1]!_{q}}}\frac{[a_{r,1}]!_{q}[a_{r-1,1}-a_{r,2}-1]!_{q}}
{[a_{r,2}]!_{q}[a_{r,1}-a_{r,2}-1]!_{q}[a_{r,1}-a_{r,2}]!_{q}[a_{r-1,1}-a_{r,1}]!_{q}}
\\&\prod_{i=1}^{r-1}{\frac{[a_{i,1}]!_{q}}{[a_{i-1,1}-a_{i,2}]!_{q}}\frac{[a_{i-1,1}-a_{i,2}-1]!_{q}}{[a_{i+1,1}]!_{q}}\frac{1}{[a_{i,2}-a_{i+1,1}]!_{q}[a_{i,1}-a_{i,2}-1]!_{q}}\frac{\frac{[a_{i-1,1}-a_{i,2}]!_{q}}{[a_{i,1}-a_{i,2}]!_{q}[a_{i-1,1}-a_{i,1}]!_{q}}}{\frac{[a_{i,1}-1]!_{q}}{[a_{i+1,1}-1]!_{q}[a_{i,1}-a_{i+1,1}]!_{q}}}}
%\\&=\frac{[N]!_{q}}{[N-1]!_{q}}\frac{[a_{1,1}-1]!_{q}}{[a_{1,1}]!_{q}}\frac{[a_{r,1}]!_{q}}{[a_{r,2}]!_{q}}
%\frac{[a_{r-1,1}-a_{r,2}-1]!_{q}}{[a_{r,1}-a_{r,2}-1]!_{q}[a_{r,1}-a_{r,2}]!_{q}[a_{r-1,1}-a_{r,1}]!_{q}}
%\\&\prod_{i=1}^{r-1}{\frac{[a_{i,1}]!_{q}}{[a_{i+1,1}]!_{q}}\frac{[a_{i+1,1}-1]!_{q}}{[a_{i,1}-1]!_{q}}\frac{[a_{i,1}-a_{i+1,1}]!_{q}}{[a_{i-1,1}-a_{i,1}]!_{q}}\frac{[a_{i-1,1}-a_{i,2}-1]!_{q}}{[a_{i,2}-a_{i+1,1}]!_{q}[a_{i,1}-a_{i,2}-1]!_{q}[a_{i,1}-a_{i,2}]!_{q}}}
\\&=\frac{\frac{[N]!_{q}}{[N-1]!_{q}}}{\frac{[a_{1,1}]!_{q}}{[a_{1,1}-1]!_{q}}}\frac{[a_{r,1}]!_{q}[a_{r-1,1}-a_{r,2}-1]!_{q}}
{[a_{r,2}]!_{q}[a_{r,1}-a_{r,2}-1]!_{q}[a_{r,1}-a_{r,2}]!_{q}[a_{r-1,1}-a_{r,1}]!_{q}}
\\&\frac{[a_{1,1}]!_{q}}{[a_{r,1}]!_{q}}\frac{[a_{r,1}-1]!_{q}}{[a_{1,1}-1]!_{q}}\frac{[a_{r-1,1}-a_{r,1}]!_{q}}{[N-a_{1,1}]!_{q}}\prod_{i=1}^{r-1}{\frac{[a_{i-1,1}-a_{i,2}-1]!_{q}}{[a_{i,2}-a_{i+1,1}]!_{q}[a_{i,1}-a_{i,2}-1]!_{q}[a_{i,1}-a_{i,2}]!_{q}}}
%\\&=\frac{[N]!_{q}}{[N-1]!_{q}}\frac{1}{[a_{r,2}]!_{q}}
%\frac{[a_{r-1,1}-a_{r,2}-1]!_{q}}{[a_{r,1}-a_{r,2}-1]!_{q}[a_{r,1}-a_{r,2}]!_{q}}
%\\&\frac{[a_{r,1}-1]!_{q}}{[N-a_{1,1}]!_{q}}\prod_{i=1}^{r-1}{\frac{[a_{i-1,1}-a_{i,2}-1]!_{q}}{[a_{i,1}-a_{i,2}-1]!_{q}[a_{i,1}-a_{i,2}]!_{q}}}\prod_{i=2}^{r}{\frac{1}{[a_{i-1,2}-a_{i,1}]!_{q}}}
%\\&=\frac{(1-q^{N})}{[a_{r,2}]!_{q}}\frac{[a_{r,1}-1]!_{q}}{[N-a_{1,1}]!_{q}}\prod_{i=1}^{r}{\frac{[a_{i-1,1}-a_{i,2}-1]!_{q}}{[a_{i,1}-a_{i,2}-1]!_{q}[a_{i,1}-a_{i,2}]!_{q}}}\prod_{i=2}^{r}{\frac{1}{[a_{i-1,2}-a_{i,1}]!_{q}}}
\\&=\frac{[N]_{q} [N-a_{1,2}-1]!_{q}}{[N-a_{1,1}]!_{q} [a_{r,2}]!_{q}}
\prod_{i=1}^r 
\frac{[a_{i,1}-a_{i+1,2}-1]!_{q}}{[a_{i,1}-a_{i,2}-1]!_{q} [a_{i,1}-a_{i,2}]!_{q}}
\prod_{i=2}^r
\frac{1}{[a_{i-1,2}-a_{i,1}]!_{q}}.
\end{align*}
\end{proof}

\section*{Acknowledgements}

This research was conducted at the 2012 summer REU (Research Experience
for Undergraduates) program at the University of Minnesota, Twin Cities, and was supported by NSF grants DMS-1001933 and DMS-1148634. We would like to
thank Gregg Musiker, Pavlo Pylyavskyy, Vic Reiner, and Dennis Stanton, who directed the program, for their support, and express particular gratitude to Dennis Stanton both for introducing us to this problem and for his indispensable
guidance throughout the research process. We would further like to thank Alex Miller for his assistance in editing this report as well as Sudhir Ghorpade and Samrith Ram for their helpful comments.

\end{document}